\documentclass[a4paper,12pt]{amsart}
\usepackage{amsmath,amssymb}
\usepackage{amscd}
\newtheorem{thm}{Theorem}[section]
\newtheorem{prop}[thm]{Proposition}
\newtheorem{lemma}[thm]{Lemma}

\newtheorem{exam}[thm]{Example}
\theoremstyle{remark}

\newcommand{\id}{{\rm{id}}}
\newcommand{\Ad}{{\rm{Ad}}}

\newcommand{\BC}{\mathbf C}

\newcommand{\BK}{\mathbf K}
\newcommand{\BB}{\mathbf B}
\newcommand{\la}{\langle}
\newcommand{\ra}{\rangle}
\newcommand{\Ind}{{\rm{Ind}}}

\newcommand{\Pic}{{\rm{Pic}}}
\newcommand{\Equi}{{\rm{Equi}}}
\newcommand{\Aut}{{\rm{Aut}}}
\newcommand{\Int}{{\rm{Int}}}
\newcommand{\Ima}{{\rm{Im}}}
\newcommand{\Ker}{{\rm{Ker}}}

\newtheorem{Def}{Definition}[section]

\title{The Picard groups for unital inclusions of unital $C^*$-algebras}
\author{Kazunori Kodaka}
\address{Department of Mathematical Sciences, Faculty of Science, Ryukyu
\endgraf
University, Nishihara-cho, Okinawa, 903-0213, Japan}

\address{\sl{E-mail address}: \rm{kodaka@math.u-ryukyu.ac.jp}}

\keywords{equivalence bimodules, inclusions of $C^*$-algebras, the Picard group}

\subjclass[2010]{46L05}

\begin{document}
\maketitle
\begin{abstract}
We shall introduce the notion of the Picard group for an inclusion of $C^*$-algebras. We shall also study its basic
properties and the relation between the Picard group for an inclusion of $C^*$-algebras and the ordinary Picard group.
Furthermore, we shall give some examples of the Picard groups for unital inclusions of unital $C^*$-algebras.
\end{abstract}

\section{Introduction}\label{sec:intro}In the previous paper \cite {KT4:morita}
we introduced the notion of the strong Morita equivalence for inclusions of $C^*$-algebras.
Then in \cite [Proposition 2.3]{KT4:morita}, we showed that the strong Morita equivalence for
inclusions of $C^*$-algebras is an equivalence relation. Thus in the same way as in Brown, Green and
Rieffel \cite {BGR:linking}, we can define the Picard group for an inclusion of $C^*$-algebras.
Also, in this paper we shall study its basic properties as in \cite {BGR:linking} and \cite {Kodaka:equivariance}.
In the last section, we shall give some examples of the Picard groups for unital inclusions of unital $C^*$-algebras.
\par
For a unital $C^*$-algebra $A$,
let $M_n (A)$ be the $n\times n$-matrix
algebra over $A$ and $I_n$ denotes the unit element in $M_n (A)$. We identify $M_n (A)$ with
$A\otimes M_n (\BC)$.
\par
Let $A$ and $B$ be $C^*$-algebras and $X$ an $A-B$-bimodule. We denote its left $A$-action and
right $B$-action on $X$ by $a\cdot x$ and $x\cdot b$ for any $a\in A$, $b\in B$, $x\in X$, respectively.
Also, we denote by $\widetilde{X}$ the dual $B-A$-bimodule of $X$ and we denote by $\widetilde{x}$
the element in $\widetilde{X}$ induced by $x\in X$.
\par
For each $C^*$-algebra $A$, let $M(A)$ be its multiplier $C^*$-algebra and for any automorphism
$\alpha$ of $A$,
let $\underline{\alpha}$ be the automorphism of $M(A)$ induced by $\alpha$.

\section{Preliminaries}\label{sec:pre} In this section, we give several notations and basic facts used in
the paper.
\par
Let $\alpha$ be an automorphism of a $C^*$-algebra $A$. Following
\cite {BGR:linking} but
in a slightly different way, we construct an $A-A$-equivalence bimodule $X_{\alpha}$ induced by
$\alpha$ as following: Let $X_{\alpha}=A$ as $\BC$-vector spaces and the left $A$-action on
$X_{\alpha}$ and the left $A$-valued inner product are defined in the usual way, but we define the
right $A$-action on $X_{\alpha}$ by $x\cdot a=x\alpha(a)$ and the right $A$-valued inner product by
$\la x, y\ra_A =\alpha^{-1}(x^* y)$ for any $a\in A$, $x, y\in X_{\alpha}$. We call $X_{\alpha}$ the
$A-A$-equivalence bimodule induced by $\alpha$.

Let $A\subset C$ be a unital inclsion of unital $C^*$-algebras and $E^A$ a conditional expectation
from $C$ onto $A$.

\begin{Def}\label{def:pre1}(\cite [Definition 1.2.2 and Lemma 2.1.6]{Watatani:index}) A finite set
$\{(u_i , u_i^* )\}_{i=1}^n \subset C\times C$ is called a
\sl
quasi-basis
\rm
for $E^A$ if
$$
\sum_{i=1}^n u_i E^A (u_i^* c)=c=\sum_{i=1}^n E^A (cu_i )u_i^*
$$
for any $c\in C$. Also, we say that $E^A$ is of
\sl
Watatani index-finite type
\rm
if there exists a quasi-basis for $E^A$ and in this case we define $\Ind_W (E^A)$,
\sl
Watatani index
\rm
of $E^A$ by
$$
\Ind_W (E^A )=\sum_{i=1}^n u_i u_i^* \in C .
$$
\end{Def}
By \cite [Proposition 1.2.8]{Watatani:index}, $\Ind_W (E^A )$ is an invertible element in $C' \cap C$ and it
does not depend on the choice of quasi-bases.

\begin{prop}\label{prop:pre2}{\rm (\cite [Proposition 1.4.1]{Watatani:index})} With the above notation,
we suppose that $E^A$ is of Watatani index-finite type. If there is another conditional expectation $F^A$
from $C$ onto $A$, there is the unique element $h\in A' \cap C$ with $E^A (h)=1$ such that
$$
F^A (c)=E^A (hc)
$$
for any $c\in C$.
\end{prop}

Under the same situation as above, we regard $C$ as a right Hilbert $A$-module as follows:
We define the right $A$-action on $C$ by
$x\cdot a=xa$ and the right $A$-valued inner product by
$\la x, y \ra_A =E^A (x^* y)$ for any $a\in A$, $x, y\in C$. Let $\BB_A (C)$ be the $C^*$-algebra of
all adjointable right $A$-module operators on $C$ and $\BK_A (C)$ the $C^*$-algebra of all
adjointable right $A$-module ``compact" operators on $C$. We regard $E^A$ as an element in $\BB_A (C)$.
We denote it by $e_A$ and we call $e_A$ the 
\sl
Jones projection
\rm
for $E^A$. Also, we regard $c\in C$ as an element in $\BB_A (C)$ by the left multiplication in $C$.

\begin{Def}\label{def:pre3}(\cite[Definition 2.1.2]{Watatani:index}) Let $C_1$ be the closure of the linear span
of $\{ce_A d\in\BB_A (C) \, | \, c,d\in C \}$. We call the $C^*$-algebra $C_1$ the
\sl
$C^*$-basic construction
\rm
for $E^A$.
\end{Def}
By the definition of $C_1$, $C_1$ is strongly Morita equivalent to $A$ with respect to the equivalence
bimodule $C$. By \cite [Lemmas 2.1.3 and 2.1.6]{Watatani:index}, $C_1 =\BK_A (C)=\BB_A (C)$ and
$C_1$ is the linear span of $\{ce_A d\in\BB_A (C) \, | \, c,d\in C \}$.

\begin{Def}\label{def:pre4}(\cite [Definition 2.3.2]{Watatani:index}) Let $E^C$ be the linear map
from $C_1$ onto $C$ defined by
$$
E^C (ce_A d)=\Ind_W (E^A )^{-1}cd
$$
for any $c, d\in C$. Then $E^C$ is a conditional expectation from $C_1$ onto $C$
and we call $E^C$ the
\sl
dual conditional expectation
\rm
of $E^A$.
\end{Def}

By \cite [Proposition 2.3.4]{Watatani:index}, $E^C$ is of Watatani index-finite type.
Let $\{(u_i , u_i^* )\}_{i=1}^n$ be a quasi-basis for $E^A$ and we set
$w_i =u_i e_A (\Ind_W (E^A ))^{\frac{1}{2}}$ for $i=1,2,\dots , n$. Then
the finite set $\{(w_i , w_i^* )\}_{i=1}^n$ is a quasi-basis
for $E^C$.

\begin{Def}\label{def:pre5} (\cite [Definition 2.1]{KT4:morita}) Inclusions of $C^*$-algebras $A\subset C$
and $B\subset D$ with $\overline{AC}=C$ and $\overline{BD}=D$ are
\sl
strongly Morita equivalent
\rm
if there are a $C-D$-equivalence bimodule $Y$ and its closed subspace $X$ satisfying the following
conditions:
\newline
(1) $a\cdot x\in X$, ${}_C \la x, y \ra\in A$ for any $a\in A$, $x, y\in X$ and $\overline{{}_C \la X, X \ra}=A$,
$\overline{{}_C \la Y, X \ra}=C$,
\newline
(2) $x\cdot b\in X$, $\la x, y \ra_B \in B$ for any $b\in B$, $x, y\in X$ and $\overline{\la X, X \ra_D }=B$,
$\overline{\la Y, X \ra_D }=D$.
\end{Def}
Then we say that $A\subset C$ is
\sl
strongly Morita equivalent to
\rm
$B\subset D$ with respect to the $C-D$-equivalence bimodule $Y$ and its closed subspace
$X$. We note that $X$ is regarded as an $A-B$-equivalence bimodule and that if strong
Morita equivalence inclusions $A\subset C$ and $B\subset D$ are unital, we do not need
to take the closure in Definition \ref {def:pre5}.

Let $A\subset C$ and $B\subset D$ be as above. Let $E^A$ and $E^B$ be conditional expectations
from $C$ and $D$ onto $A$ and $B$, respectively. Let $E^X$ be a linear map from $Y$ onto $X$.

\begin{Def}\label{def:pre6}(\cite [Definition 2.4]{KT4:morita}) We call $E^X$ a
\sl
conditional expectation
\rm
from $Y$ onto $X$ with respect to $E^A$ and $E^B$ if $E^X$ satisfies the following:
\newline
(1) $E^X (c\cdot x)=E^A (c)\cdot x$ for any $c\in C$, $x\in X$,
\newline
(2) $E^X (a\cdot y)=a\cdot E^X (y)$ for any $a\in A$, $y\in Y$,
\newline
(3) $E^A ({}_C \la y, x \ra)={}_C \la E^X (y) , x \ra $ for any $x\in X$, $y\in Y$,
\newline
(4) $E^X (x\cdot d)=x\cdot E^B (d)$ for any $d\in D$, $x\in X$,
\newline
(5) $E^X (y\cdot b)=E^X (y)\cdot b$ for any $b\in B$, $y\in Y$,
\newline
(6) $E^B (\la y, x \ra_D )=\la E^X (y), x \ra_D$ for any $x\in X$, $y\in Y$.
\end{Def}

With the above notation, we have the following by \cite [Theorem 2.9]{KT4:morita}.

\begin{thm}\label{thm:pre7} We suppose that inclusions $A\subset C$ and $B\subset D$ are strongly Morita
equivalent with respect to a $C-D$-equivalence bimodule $Y$ and its closed subspace $X$.
If there is a conditional expectation $E^A$ of Watatani index-finite type from $C$ onto $A$, then
there are a conditional expectation of Watatani index-finite type from $D$ onto $B$ and a
conditional expectation $E^X$ from $Y$ onto $X$ with respect to $E^A$ and $E^B$. Also,
if there is a conditional expectation $E^B$ of Watatani index-finite type from $D$ onto $B$, then
we have the same result as above.

\end{thm}

Next, we mention on the upward basic construction of equivalence bimodules.
Let $A\subset C$, $B\subset D$ and $X\subset Y$ be as above. We suppose that there are
conditional expectation $E^A$ and $E^B$ from $C$ and $D$ onto $A$ and $B$, respectively and
that they are of Watatani index-finite type. We also suppose that there is a conditional expectation $E^X$
from $Y$ onto $X$ with respect to $E^A$ and $E^B$. Let $C_1$ and $D_1$ be the $C^*$-basic
constructions for $E^A$ and $E^B$, respectively. And let $E^C$ and $E^D$ be
the dual conditional expectations
of $E^A$ and $E^B$, respectively. We regard $C$ and $D$ as a $C_1 - A$-equivalence
bimodule and a $D_1 -B$-equivalence bimodule, respectively. Let
$$
Y_1 =C\otimes_A X\otimes_B\widetilde{D} .
$$
Also, let $E^Y$ be the linear map from $Y_1$ onto $Y$ defined by
$$
E^Y (c\otimes x\otimes\widetilde{d})=\Ind_W (E^A )^{-1}c\cdot x\cdot d^*
$$
for any $c\in C$, $d\in D$, $x\in X$. Furthermore, let $\phi$ be the linear map from $Y$ to $Y_1$
defined by
$$
\phi(y)=\sum_{i, j}u_i \otimes E^X (u_i^* \cdot y\cdot v_j )\otimes\widetilde{v_j}
$$
for any $y\in Y$, where $\{(u_i , u_i^* )\}$ and $\{(v_j , v_j^* )\}$ are quasi- bases for $E^A$ and
$E^B$, respectively. By \cite [Lemma 6.1 and Corollary 6.3]{KT4:morita}, we obtain the following:

\begin{thm}\label{thm:pre8} $($\cite [Lemma 6.1 and Corollary 6.3]{KT4:morita}$)$
We can regard $Y$ as a closed subspace of
a $C_1- D_1$-equivalence bimodule $Y_1$ using the map $\phi$ and
the inclusions $C\subset C_1$ and $D\subset D_1$ are strongly Morita equivalent with $Y_1$
and its closed subspace $Y$.
\end{thm}
By \cite [Lemma 6.4 and Remark 6.6]{KT4:morita}, we can see that
$E^Y$ is a conditional
expectation from $Y_1$ onto $Y$ with respect to $E^C$ and $E^D$ and it is independent of the
choice of quasi-bases for $E^A$ and $E^B$.

\begin{Def}\label{def:pre9}(\cite [Definition 6.5]{KT4:morita}) We call $Y_1$ the
\sl
upward basic construction
\rm
of $Y$ for $E^X$, and $E^Y$ is called
\sl
the dual conditional expectation
\rm
of $E^X$.
\end{Def}

\section{Definitions and basic properties}\label{sec:definition}Let $A\subset C$, $B\subset D$ and
$K\subset L$ be inclusions of $C^*$-algebras with $\overline{AC}=C$, $\overline{BD}=D$ and
$\overline{KL}=L$, respectively. Let $Y$ and $W$ be a $C-D$-equivalence bimodule and a $D-L$-equivalence
bimodule and $X$ and $Z$ their closed subspaces satisfying Conditions (1), (2) in
Definition \ref{def:pre5}, respectively. That is, the inclusions $A\subset C$ and $B\subset D$
are strongly Morita equivalent with respect to the $C-D$-equivalence bimodule $Y$ and
its closed subspace $X$ and the inclusion $B\subset D$ and $K\subset L$ are strongly Morita
equivalent with respect to the $D-L$-equivalence bimodule and its closed subspace $Z$.
Let $X\otimes_D Z$ be the closure of the linear span of the set
$$
\{x\otimes z \in Y\otimes_D W \, | \, x\in X, \, z\in Z \} .
$$
Clearly $X\otimes_D Z$ is a closed linear subspace of $Y\otimes_D W$ and we can regard $X\otimes_D Z$ as
an $A-K$-equivalence bimodule.

\begin{lemma}\label{lem:bimodule}With the above notation, $X\otimes_D Z$ is isomorphic to
$X\otimes_B Z$ as $A-K$-equivalence bimodules, where $X\otimes_D Z$ is regarded as an $A-K$-equivalence bimodule.
\end{lemma}
\begin{proof}Let $(X\otimes_D Z)_0$ be the linear span of the set
$$
\{x\otimes z \in Y\otimes_D W \, | \, x\in X, \, z\in Z \} 
$$
and let $(X\otimes_B Z)_0$ be the algebraic relative tensor product of the $A-B$-equivalence bimodule $X$
and the $B-K$-equivalence bimodule $Z$. We note that $X\otimes_B Z$ is the completion of $(X\otimes_B Z)_0$.
Let $\pi_0$ be the map from $(X\otimes_B Z)_0$ to $(X\otimes_D Z)_0$ defined by
$\pi_0 (x\otimes_B z)=x\otimes_D z$ for any $x\in X$, $z\in Z$. It is well-defined and surjective.
By easy computations, $\pi_0$ preserves the left $A$-valued inner product and the right $K$-valued
inner product. Hence we obtain an $A-K$-equivalence bimodule isomorphism $\pi$ of $X\otimes_B Z$
onto $X\otimes_D Z$.
Therefore, we obtain the conclusion.
\end{proof}
By Lemma \ref{lem:bimodule}, we identify $X\otimes_B Z$ with $X\otimes_D Z$, the closed subspace of
$Y\otimes_D W$ under the above situations.

Let $A\subset C$ be an inclusion of $C^*$-algebras with $\overline{AC}=C$.
Let $Y$ be a $C-C$-equivalence bimodule and $X$ its closed subspace satisfying
Conditions (1), (2) in Definition \ref {def:pre5}. Let $\Equi(A, C)$ be the set of all
such pairs $(X, Y)$ as above. We define an equivalence relation $`` \sim "$ as follows:
For $(X, Y)$, $(Z, W)\in\Equi(A, C)$, $(X, Y)\sim (Z, W)$ in $\Equi (A, C)$ if and only if
there is a $C-C$-equivalence bimodule isomorphism $\Phi$ of $Y$ onto $W$ such that
the restriction of $\Phi$ to $X$, $\Phi|_X$ is an $A-A$-equivalence bimodule isomorphism of $X$
onto $Z$. We denote by $[X, Y]$, the equivalence class of $(X, Y)$ in $\Equi(A, C)$.
We remark here that we have the following lemma:

\begin{lemma}\label{lem:simple}We suppose that inclusions of $C^*$-algebras $A\subset C$
and $B\subset D$ are strongly Morita equivalent with respect to $C-D$-equivalence bimodules
$Y$ and $W$ and their closed subspaces $X$ and $Z$, respectively.
If there is a $C-D$-equivalence bimodule isomorphism $\Phi$ of $Y$ onto $W$
such that $\Phi|_X$ is a bijection from $X$ onto $Z$, then $\Phi|_X$ is an $A-B$-
equivalence bimodule isomorphism of $X$ onto $Z$.
\end{lemma}
\begin{proof}Let $a\in A$, $b\in B$ and $x, y\in X$. By Definition \ref {def:pre5},
$a\cdot x\in X$ and since $\Phi|_X$ is a bijection from $X$ onto $Z$,
$\Phi(a\cdot x)\in Z$. Furthermore, since $\Phi$ is a $C-D$-equivalence bimodule isomorphism
of $Y$ onto $W$, $\Phi(a\cdot x)=a\cdot \Phi(x)$. Similarly we obtain that $\Phi(x\cdot b)=\Phi(x)\cdot b$.
Also by Definition \ref {def:pre5},
\begin{align*}
{}_A \la \Phi(x), \, \Phi(y) \ra & ={}_C \la \Phi(x) \, , \, \Phi(y) \ra ={}_C \la x, y \ra
={}_A \la x, y \ra  \\
\la \Phi(x) , \, \Phi(y) \ra_B  & =\la \Phi(x) \, , \, \Phi(y) \ra_D
= \la x , \, y \ra_D =\la x \, , \, y \ra_B .
\end{align*}
Hence $\Phi|_X$ is an $A-B$-equivalence bimodule isomorphism of
$X$ onto $Z$.
\end{proof}

By Lemma \ref{lem:simple}, we can see that for $(X, Y), (Z, W)\in\Equi(A, C)$, $(X, Y)\sim(Z, W)$
in $\Equi(A, C)$ if and only if there is a $C-C$-equivalence bimodule isomorphism $\Phi$ of $Y$ onto $W$
such that $\Phi(X)=Z$.
\par
Let $\Pic(A, C)=\Equi (A, C)/\!\sim$. We define the product in $\Pic(A, C)$ as follows: For $(X, Y)$,
$(Z, W)\in\Equi(A, C)$
$$
[X, Y][Z, W]=[X\otimes_A Z \, , \, Y\otimes_C W],
$$
where the $A-A$-equivalence bimodule $X\otimes_A Z$ is identified with the closed
subspace $X\otimes_C Z$ of $Y\otimes_C W$ defined in the above. We note that
$Y\otimes_C W$ and its closed subspace $X\otimes_A Z$ satisfy Conditions (1), (2) in
Definition \ref {def:pre5} by \cite [Proposition 2.3]{KT4:morita}. By Lemma \ref{lem:bimodule}
and easy computations, we can see that $\Pic(A, C)$ is a group. We regard $(A, C)$ as an
element in $\Equi(A, C)$ in the evident way. Then $[A, C]$ is the unit element in $\Pic(A, C)$.
For any element $(X, Y)\in\Equi(A, C)$, $(\widetilde{X}, \widetilde{Y})\in\Equi(A, C)$ and
$[\widetilde{X}, \widetilde{Y}]$ is the inverse element of $[X, Y]$ in $\Pic(A, C)$, where $\widetilde{X}$
and $\widetilde{Y}$ are the dual $A-A$-equivalence bimodule of $X$ and the dual $C-C$-equivalence bimodule
of $Y$, respectively. We note that $\widetilde{X}$ can be a closed subspace of $\widetilde{Y}$.
We call the group $\Pic(A, C)$ defined in the above, the
\sl
Picard group
\rm
of the inclusion of $C^*$-algebras $A\subset C$ with $\overline{AC}=C$.
\par
Let $A\subset C$ and $B\subset D$ be inclusions of $C^*$-algebras with $\overline{AC}=C$
and $\overline{BD}=D$, respectively. We suppose that $A\subset C$ and $B\subset D$ are
strongly Morita equivalent with respect to a $C-D$-equivalence $W$ and its closed subspace $Z$.
Let $g$ be the map from $\Pic(A, C)$ to $\Pic(B, D)$ defined by
$$
g([X, Y])=[\widetilde{Z}\otimes_A X\otimes_A  Z \, , \, \widetilde{W}\otimes_C Y\otimes_C W]
$$
for any $(X, Y)\in\Equi(A, C)$, where $\widetilde{Z}\otimes_A X\otimes_A Z$ is regarded as a 
closed subspace of $\widetilde{W}\otimes_C Y\otimes_C W$ in the same way as in Lemma \ref{lem:bimodule}.

\begin{lemma}\label{lem:iso1}With the same notation as above, $g$ is an isomorphism of $\Pic(A, C)$
onto $\Pic(B, D)$,
\end{lemma}
\begin{proof}The $C-C$-equivalence bimodule $W\otimes_D \widetilde{W}$ is isomorphic
to the $C-C$-equivalence bimodule $C$ by the isomorphism
$$
W\otimes_D \widetilde{W}\longrightarrow C:z\otimes\widetilde{w}\mapsto{}_C \la z, w \ra
$$
for any $z, w\in W$. The restriction of the above isomorphism to $Z\otimes_B \widetilde{Z}$ is
an isomorphism of $Z\otimes_B \widetilde{Z}$ onto $A$ by Definition \ref{def:pre5}. 
Also, the $C-C$-equivalence bimodule $C\otimes_C Y$ is isomorphic to the $C-C$-equivalence bimodule
$Y$ by the isomorphism
$$
C\otimes_C Y\longrightarrow Y : c\otimes y\mapsto c\cdot y
$$
for any $c\in C$, $y\in Y$. The restriction of the above isomorphism to $A\otimes_A X$ is
an isomorphism of $A\otimes_A X$ onto $X$ by Definition \ref{def:pre5}.
By the above discussions, we can see that $g$ is an isomorphism of $\Pic(A, C)$
onto $\Pic(B, D)$.
\end{proof}

Let $\alpha$ be an automorphism of $C$ such that the restriction of $\alpha$ to $A$,
$\alpha|_A$ is an automorphism of $A$. Let $\Aut (A, C)$ be the group of all such
automorphisms. We construct an element in $\Equi (A, C)$ from an element in $\Aut (A, C)$ as follows:
Let $\alpha\in\Aut(A, C)$. Let $Y_{\alpha}$ be the $C-C$-equivalence bimodule
induced by $\alpha$ in the same way as in Preliminaries.
Let $X_{\alpha}$ be the $A-A$-equivalence bimodule induced by $\alpha |_A$ in the same way
as above. Then clearly $(X_{\alpha}, Y_{\alpha})\in \Equi(A, C)$ and for any $\alpha$, $\beta\in\Aut (A, C)$,
$$
[X_{\alpha\circ\beta} \, , \, Y_{\alpha\circ\beta}]=[X_{\alpha} \, , \,Y_{\alpha}][X_{\beta} \, , \, Y_{\beta}]
$$
in $\Pic(A, C)$. Let $\pi$ be the map from $\Aut (A, C)$ to $\Pic(A, C)$ defined by
$$
\pi(\alpha)=[X_{\alpha} \, , \, Y_{\alpha}]
$$
for any $\alpha\in \Aut(A, C)$. By the above discussions, $\pi$ is a homomorphism of
$\Aut (A, C)$ to $\Pic(A, C)$. Let $u$ be a unitary element in $M(A)$. Then $\Ad(u)$ is a
generalized inner automorphism of $A$. Since $\overline{AC}=C$,
by Izumi \cite {Izumi:simple} $u\in M(C)$. Thus $\Ad(u)$ is also
a generalized inner automorphism of $C$. Let $\Int (A, C)$ be the set of all such automorphisms
in $\Aut (A, C)$. We note that $\Int (A, C)=\Int (A)$.
Let $\imath$ be the inclusion map of $\Int (A, C)$ to $\Aut (A, C)$.

\begin{lemma}\label{lem:exact1}With the above notation, the sequence
$$
1\longrightarrow \Int (A, C)\overset{\imath}\longrightarrow\Aut (A, C)\overset{\pi}\longrightarrow\Pic(A, C)
$$
is exact.
\end{lemma}
\begin{proof}Let $u\in M(A)$. Then $u\in M(C)$. We show that $C\cong Y_{\Ad(u)}$ as $C-C$-equivalence
bimodules. Let $\Phi$ be the map from $C$ to $Y_{\Ad(u)}$ defined by
$\Phi(x)=xu^*$ for any $x\in C$. Then for any $a\in C$, $x, y\in C$,
\begin{align*}
\Phi(a\cdot x) & =\Phi(ax)=axu^* =a\cdot \Phi(x) \\
\Phi(x\cdot a) & =xu^* uau^* =\Phi(x)\cdot a \\
{}_C \la \Phi(x) , \Phi(y) \ra & ={}_C \la xu^* \, , yu^* \ra =xy^* = {}_C \la x, y \ra \\
\la \Phi(x) , \Phi(y) \ra_C & =\la xu^* , yu^* \ra_C =\Ad(u)^{-1}(ux^* yu^* )=x^* y ={} \la x, y \ra_C .
\end{align*}
Hence $\Phi$ is a $C-C$-equivalence bimodule isomorphism of $C$ onto $Y_{\Ad(u)}$.
Furthermore, since $u\in M(A)$, $xu^* \in X_{\Ad(u)}$ for any $x\in A$. In the same way
as above, we can see that $\Phi|_{\Ad(u)}$ is an $A-A$-equivalence bimodule isomorphism
of $A$ onto $X_{\Ad(u)}$. Thus 
$$
[X_{\Ad(u)} \, , \, Y_{\Ad(u)} ]=[A, C]
$$
in $\Pic(A, C)$. Let $\alpha\in\Aut (A, C)$ with $[X_{\alpha}, Y_{\alpha}]=[A, C]$ in
$\Pic(A, C)$. Then there is a $C-C$-equivalence bimodule isomorphism $\Phi$ of $C$
onto $Y_{\alpha}$ such that $\Phi|_A$ is an $A-A$-equivalence bimodule isomorphism
of $A$ onto $X_{\alpha}$. In the same way as the proof of \cite [Proposition 3.1]{BGR:linking},
we can obtain unitary elements $u_1 \in M(C)$ and $u\in M(A)$ such that
\begin{align*}
u_1  =(\Phi\circ\alpha^{-1} \, , \, \Phi) \,, &  \quad u=((\Phi\circ\alpha^{-1})|_A \, , \, \Phi|_A ) \\
\alpha =\Ad(u_1^* ) \, , &  \quad \alpha|_A =\Ad(u^* ) ,
\end{align*}
where $(\Phi\circ\alpha^{-1} , \Phi)$ and $((\Phi\circ\alpha^{-1})|_A , \, \Phi|_A )$
are double centralizers of $C$ and $A$, respectively.
Then for any $a\in A$, $u_1 a=(\Phi\circ\alpha^{-1})(a)=ua$. Since $\overline{AC}=C$,
$u_1 =u$. Hence $\pi([X_{\Ad(u^* )} \, , Y_{\Ad(u^* )}])=[A, C]$. Therefore, we obtain the
conclusion.
\end{proof}

Let $A\subset C$ be an inclusion of $C^*$-algebras such that $A$ is $\sigma$-unital
and $\overline{AC}=C$. Let $\BK$ be the $C^*$-algebra of
all compact operators on a countably infinite dimensional Hilbert space and let $A^s =A\otimes \BK$
and $C^s =C\otimes\BK$, respectively. Let $(X, Y)\in\Equi (A^s, C^s )$.
Let $L_X$ and $L_Y$ be the linking $C^*$-algebras induced by $X$ and $Y$, respectively.
Let
$$
p=\begin{bmatrix}1_A \otimes 1_{M(\BK)} & 0 \\
0 & 0 \end{bmatrix}, \quad
q=\begin{bmatrix}0 & 0 \\
0 & 1_A \otimes 1_{M(\BK)} \end{bmatrix}
$$
in $M(L_X )$. Then $p$ and $q$ are full projections in $M(L_X )$. By easy computations,
we can see that $\overline{L_X L_Y }=L_Y$. Hence $M(L_X )\subset M(L_Y )$ by
Izumi \cite {Izumi:simple}. Since $p$ and
$q$ are full projections in $M(L_X )$,  by Brown \cite [Lemma 2.5]{Brown:hereditary}, 
there is a partial isometry $w\in M(L_X )$ such that $w^* w=p$, $ww^* =q$. Then we note
that $w\in M(L_Y )$. Let $\theta$
be the map from $pL_Y p$ to $qL_Y q$ defined by
$$
\theta(\begin{bmatrix} x & 0 \\
0 & 0 \end{bmatrix})=w\begin{bmatrix} x & 0 \\
0 & 0 \end{bmatrix}w^*
$$
for any $x\in C^s$. By easy computations, we can see that $\theta$ is an isomorphism of $pL_Y p$
onto $qL_Y q$. Identifying $pL_Y p$ and $qL_Y q$ with $C^s$, we can regard $\theta$ as an
automorphism of $C^s$. We also denote it by the same symbol $\theta$. Since $w\in M(L_X )$,
$\theta|_{A^s}$ is an automorphism of $A^s$. Let $(X_{\theta}, Y_{\theta})$ be the element in $\Equi(A, C)$
induced by $\theta$. Then we can see that $[X_{\theta}, Y_{\theta}]=[X, Y]$ in $\Pic(A, C)$ in the
same way as in the proof of \cite [Theorem 3.4]{BGR:linking}. By the above discussions and
Lemma \ref{lem:exact1}. we obtain the following proposition:

\begin{prop}\label{prop:exact2}With the above notation and assumptions, the sequence
$$
1\longrightarrow \Int (A^s , C^s )\overset{\imath}\longrightarrow\Aut (A^s , C^s )\overset{\pi}\longrightarrow
\Pic(A^s , C^s )\longrightarrow 1
$$
is exact.
\end{prop}

\section{Some lemmas}\label{sec:lemmas}In this section, we shall prepare some lemmas for the next sections.
Let $A\subset C$, $B\subset D$ and $K\subset L$ be unital inclusions of unital $C^*$-algebras.
Let $E^A$, $E^B$ and $E^K$ be conditional expectations from $C$, $D$ and $L$ onto $A$, $B$ and $K$,
respectively. We suppose that they are of Watatani index-finite type. Also, we suppose that $A\subset C$ and
$B\subset D$ are strongly Morita equivalent with respect to a $C-D$-equivalence bimodule $Y$ and its
closed subspace $X$ and suppose  that $B\subset D$ and $K\subset L$ are strongly Morita equivalent with
respect to a $D-L$-equivalence bimodule $W$ and its closed subspace $Z$.  Also, we suppose that
$A' \cap C=\BC1$. Then $B' \cap D=\BC 1$ and $K' \cap L=\BC1$ by \cite [Lemma 10.3]{KT4:morita}.

\begin{lemma}\label{lem:expectation}With the above notation and assumptions,
there is the unique conditional expectation $E^X$ from
$Y$ onto $X$ with respect to $E^A$ and $E^B$.
\end{lemma}
\begin{proof}By Theorem \ref {thm:pre7}, we can see that there are a conditional expectation $F^B$
of Watatani index-finite type from $D$ onto $B$ and a conditional expectation $E^X$ from
$Y$ onto $X$ with respect to $E^A$ and $F^B$. But $F^B=E^B$
by Proposition \ref {prop:pre2}
since $A' \cap C=\BC1$. Hence
$E^X$ is a conditional expectation from $Y$ onto $X$ with respect to $E^A$ and $E^B$.
Next we show the uniqueness of $E^X$. Let $F^X$ be another conditional expectation from
$Y$ onto $X$ with respect
to $E^A$ and $E^B$. Then by the definitions of $E^X$ and $F^X$, for any $x\in X$, $y\in Y$,
$$
\la x \, , \, E^X (y) \ra_B =E^B (\la x , y \ra_D )= \la x \, , \, F^X (y) \ra_B .
$$
Hence $E^X (y)=F^X (y)$ for any $y\in Y$.
\end{proof}

By Lemma \ref{lem:expectation}, there is the unique conditional expectation $E^{\widetilde{X}}$ from
$\widetilde{Y}$ onto $\widetilde{X}$ with respect to $E^B$ and $E^A$.

\begin{lemma}\label{lem:conjugate}With the above notation and assumptions, $E^{\widetilde{X}}(\widetilde{y})
=\widetilde{E^X (y)}$ for any $y\in Y$.
\end{lemma}
\begin{proof}This is immediate by Definition \ref{def:pre6} and routine computations.
\end{proof}

Also, by Lemma \ref {lem:expectation}, there are the unique conditional expectations $E^X$ and
$E^Z$ from $Y$ and $W$ onto $X$ and $Z$ with respect to $E^A$, $E^B$ and $E^B$, $E^K$,
respectively. Also, there is the unique conditional expectation $E^{X\otimes_B Z}$ from $Y\otimes_D W$
onto $X\otimes_B Z$ with respect to $E^A$ and $E^K$. Let $C_1$, $D_1$ and $L_1$ be the
$C^*$-basic constructions for $E^A$, $E^B$ and $E^K$, respectively. Let $Y_1$ and $W_1$
be the upward basic constructions of $Y$ and $W$ for $E^X$ and $E^Z$, respectively and let
$(Y\otimes_D W)_1$ be the upward basic construction of $Y\otimes_D W$ for $E^{X\otimes_B Z}$.

\begin{lemma}\label{lem:prepare2}With the above notation and asumptions,
$(Y\otimes_D W)_1 \cong Y_1 \otimes_{D_1}W_1$ as $C_1 -D_1$-equivalence bimodules.
\end{lemma}
\begin{proof}
By the definitions of $Y_1$, $W_1$ and $(Y\otimes_D W)_1$,
\begin{align*}
Y_1 =C\otimes_A  & X\otimes_B \widetilde{D} \, , \quad W_1 =D\otimes_B Z \otimes_K \widetilde{L} , \\
(Y\otimes_D W)_1 & =C\otimes_A X\otimes_B Z\otimes_K \widetilde{L}
\end{align*}
Thus
$$
Y_1 \otimes_{D_1}W_1 =C\otimes_A X \otimes_B \widetilde{D}\otimes_{D_1}D\otimes_B Z\otimes_K\widetilde{L} .
$$
Let $\Phi$ be the map from $Y_1 \otimes_{D_1}W_1$ to $(Y\otimes_D W)_1$ defined by
$$
\Phi(c\otimes x\otimes\widetilde{d}\otimes d' \otimes z\otimes\widetilde{l})=c\otimes x \otimes E^B (d^* d' )\cdot z\otimes \widetilde{l}
$$
for any $c\in C$, $x\in X$, $d, d' \in D$, $z\in Z$, $l\in L$. By routine computations, $\Phi$ is a $C_1 -D_1$-equivalence bimodule isomorphism of $Y_1 \otimes_{D_1}W_1$ onto $(Y\otimes_D W)_1$. Therefore, we obtain the conclusion.
\end{proof}

Let $\{(u_i , u_i^* )\}$, $\{(v_i , v_i^* )\}$ and $\{(s_i , s_i^* )\}$ be quasi-bases for conditional expectations
$E^A$, $E^B$ and $E^K$, respectively. We recall that $Y$ and $W$ are regarded as closed subspaces of $Y_1$
and $W_1$ by the linear maps $\phi_Y$ and $\phi_W$ defined by
\begin{align*}
\phi_Y (y) & =\sum_{i, j}u_i\otimes E^X (u_i^* \cdot y\cdot v_j )\otimes\widetilde{v_j} , \\
\phi_W (w) & =\sum_{k, l}v_k \otimes E^Z (v_k^* \cdot w\cdot s_l )\otimes\widetilde{s_l}
\end{align*}
for any $y\in Y$, $w\in W$, respectively.

\begin{lemma}\label{lem:prepare3}With the above notation and asssumptions,
$$
E^{X\otimes_B Z}(y\otimes w)=\sum_j E^X (y\cdot v_j )\otimes E^Z (v_j^* \cdot w)
$$ 
for any $y\in Y$, $w\in W$.
\end{lemma}
\begin{proof}By Lemma \ref {lem:expectation}, we have only to show the right hand side
in the above equation defines a linear map
satisfying Conditions (1)-(6) in Definition \ref {def:pre6}. They are proved in the routine computations.
\end{proof}

Also, we recall that $Y\otimes_D W$ is regarded as a closed subspace of $(Y\otimes_D W)_1$ by the
linear map $\phi_{Y\otimes_D W}$ defined by
$$
\phi_{Y\otimes_D W}(y\otimes w)=\sum_{i, l}u_i \otimes E^{X\otimes_B Z}(u_i \cdot y\otimes w\cdot s_l)
\otimes\widetilde{s_l}
$$
for any $y\in Y$, $w\in W$.

\begin{lemma}\label{lem:prepare4}With the above notation, let $\Phi$ be the $C_1 -L_1$-equivalence
bimodule isomorphism of $Y_1 \otimes_{D_1}W_1$ onto $(Y\otimes_D W)_1$ defined in Lemma
\ref {lem:prepare2}. Then
$$
\phi_{Y\otimes_D W}=\Phi\circ (\phi_Y \otimes \phi_W ) .
$$
\end{lemma}
\begin{proof}For any $y\in Y$, $w\in W$,
$$
\phi_Y (y)\otimes\phi_W (w)=\sum_{i,j,t,l}u_i \otimes E^X (u_i^* \cdot y\cdot v_j )\otimes\widetilde{v_j}
\otimes v_t \otimes E^Z (v_t^* \cdot w\cdot s_l )\otimes\widetilde{s_l} .
$$
By the definition of $\Phi$,
\begin{align*}
& \Phi(\phi_Y (y)\otimes\phi_W (w)) \\
& =\sum_{i, j, t, l}u_i \otimes E^X (u_i^* \cdot y\cdot v_j )
\otimes E^A (v_j^* v_t )\cdot E^Z (v_t^* \cdot w\cdot s_l )\otimes\widetilde{s_l} \\
& =\sum_{i, j, t, l}u_i \otimes E^X (u_i^* \cdot y\cdot v_j )\otimes E^Z (E^A (v_j^* v_t )v_t^* \cdot w\cdot s_l )
\otimes\widetilde{s_l} \\
& =\sum_{i, j, l}u_i \otimes E^X (u_i^* \cdot y\cdot v_j )\otimes E^Z (v_j^* \cdot w\cdot s_l )\otimes\widetilde{s_l} .
\end{align*}
On the other hand, by Lemma \ref {lem:prepare3}
\begin{align*}
\phi_{Y\otimes_D W}(y\otimes w) & =\sum_{i, l}u_i \otimes E^{X\otimes_B Z}
(u_i^* \cdot y\otimes w\cdot s_l )\otimes\widetilde{s_l} \\
& =\sum_{i, j, l}u_i \otimes E^X (u_i^* \cdot y\cdot v_j )\otimes E^Z (v_j^* \cdot w\cdot s_l )\otimes\widetilde{s_l} .
\end{align*}
Therefore we obtain the conclsuion.
\end{proof}

\section{The $C^*$-basic construction}
\label{sec:construction}
Let $A\subset C$ be a unital inclusion of unital $C^*$-algebras. We suppose that $A' \cap C=\BC1$ and that 
there is a conditional expectation $E^A$ of Watatani index-finite type from $C$ onto $A$.
We denote its Watatani index by $\Ind_W (E^A )$. Then $\Ind_W (E^A )\in\BC1$.
Let $C_1$ be the $C^*$-basic
construction for $E^A$ and $e_A$ the Jones projection for $E^A$.
Let $(X, Y)\in \Equi(A, C)$.
Let $Y_1$ be the upward basic construction of $Y$ for $E^X$. Then $Y_1$ is uniquely determined
by Lemma \ref{lem:expectation}. We recall that $Y$ is regarded as a
closed subspace of $Y_1$ by the map, which is denoted by $\phi_Y$, from $Y$ to $Y_1$
defined by
$$
\phi_Y(y)=\sum_{i, j}u_i\otimes E^X (u_i^* \cdot y\cdot u_j )\otimes\widetilde{u_j} ,
$$
where $\{(u_i, u_i^* )\}_i$ is a quasi-basis for $E^A$, which is independent of the choice
of a quasi-basis for $E^A$ (See Preliminaries). Let $f$ be the map from $\Pic(A, C)$ to $\Pic(C, C_1 )$
defined by
$$
f([X, Y])=[Y, Y_1 ]
$$
for any $(X, Y)\in \Equi(A, C)$. Then by Theorem \ref {thm:pre8}, $(Y, Y_1 )\in \Equi(C, C_1 )$.
Hence $f$ is well-defined. In this section, we shall show that $f$ is an isomorphism
of $\Pic(A, C)$ onto $\Pic(C, C_1 )$. First we show that $f$ is a homomorphism of  $\Pic(A,C)$ to
$\Pic(C, C_1 )$.

\begin{lemma}\label{lem:product}With the same notation as above, $f$ is a homomorphism of $\Pic(A, C)$
to $\Pic(C, C_1 )$.
\end{lemma}
\begin{proof}Let $(X, Y)$, $(Z, W)\in \Equi(A, C)$. Then
$$
f([X, Y][Z, W])=f([X\otimes_A Z \, , \, Y\otimes_C W])=[Y\otimes_C W \, , \, (Y\otimes_C W)_1 ] ,
$$
where $(Y\otimes_C W)_1$ is the upward basic construction of $Y\otimes_C W$
for $E^{X\otimes_A Z}$ where $E^{X\otimes_A Z}$ is the conditional expectation from $Y\otimes_C W$
onto $X\otimes_A Z$ with respect to $E^A$ and $E^A$.
By Lemmas \ref{lem:prepare2}, \ref{lem:prepare4},
we can see that there is a $C_1-C_1$-equivalence bimodule 
isomorphism preserving the elements in $Y\otimes_C W$.
Therefore, we obtain that
$$
f([X, Y][Z, W])=f([X, Y])f([Z, W]) .
$$
\end{proof}

Let $E^C$ be the dual conditional expectation from $C_1$ onto $C$.
Let $e_C$ and $C_2$ be the Jones projection and the $C^*$-basic construction
for $E^C$, respectively. Then the unital inclusion $C_1 \subset C_2$ is strongly
Morita equivalent to the unital inclusion $A\subset C$ with respect to the
$C_2 -C$-equivalence bimodule $C_1$ and its closed subspace $C$ by
\cite [Lemma 4.2]{KT4:morita}, where $C$ is regarded as a closed subspace of $C_1$
by the linear map $\theta_C$ defined by
$$
\theta_C (a)=\Ind_W (E^A )^{\frac{1}{2}}ae_A
$$
for any $a\in C$. Let $g$ be the map from $\Pic(A, C)$ to $\Pic(C_1 , C_2 )$
defined by
$$
g([X, Y])=[C\otimes_A X\otimes_A \widetilde{C}\, , \, C_1 \otimes_C Y\otimes _C \widetilde{C_1}]
$$
for any $(X, Y)\in \Equi(A, C)$. Then $g$ is an isomorphism of $\Pic(A, C)$ onto $\Pic(C_1 , C_2 )$
by Lemma \ref{lem:iso1}. Let $f_1$ be the homomorphism of $\Pic(C, C_1 )$ to $\Pic(C_1 , C_2 )$
defined by
$$
f_1 ([Y, Y_1 ])=[Y_1 , Y_2 ]
$$
for any $(Y, Y_1 )\in \Equi (C, C_1 )$, where $Y_2$ is the upward basic construction
of $Y_1$ for $E^Y$ and $E^Y$ is the conditional expectation from $Y_1$ onto $Y$ with
respect to $E^C$ and $E^C$.

\begin{lemma}\label{lem:injective}With the above notation, $f_1 \circ f=g$ on $\Pic(A, C)$.
\end{lemma}
\begin{proof}Let $(X, Y)\in \Equi (A, C)$. By the definitions of $f$ and $f_1$,
$$
(f_1 \circ f)([X, Y])=[Y_1, Y_2 ],
$$
where $Y_1 = C\otimes_A X \otimes_A \widetilde{C}$ and $Y_2 =C_1 \otimes_C Y\otimes_C \widetilde{C_1}$.
Also,
$$
g([X, Y])=[C\otimes_A X\otimes_A \widetilde{C} \, , \, C_1\otimes_C Y \otimes_C \widetilde{C_1} ]
$$
by the definition of $g$. We note that $Y_1$ is regarded as a closed subspace $Y_2$ by
the linear map $\phi_{Y_1}$ from $Y_1$ to $Y_2$ defined by
$$
\phi_{Y_1} (c\otimes x\otimes\widetilde{d})
=\sum_{i, j}w_i \otimes E^Y (w_i^* \cdot c\otimes x\otimes\widetilde{d}\cdot w_j )
\otimes\widetilde{w_j}
$$
for any $c, d\in C$, $x\in X$, where $\{(w_i , w_i^* )\}$ is a quasi-basis for $E^C$ defined by
$w_i =\Ind_W (E^A )^{\frac{1}{2}}u_i e_A$. We also note that $C\otimes_A X\otimes_A \widetilde{C}$ is regarded 
as a closed subspace of $C_1\otimes_C Y \otimes_C\widetilde{C_1}$ by the linear map
$\theta_{C\otimes_A X\otimes_A \widetilde{C}}$
from $C\otimes_A X\otimes_A \widetilde{C}$ to $C_1\otimes_C Y \otimes_C\widetilde{C_1}$ defined by
$$
\theta_{C\otimes_A X\otimes_A\widetilde{C}}(c\otimes x\otimes\widetilde{d})
=\Ind_W (E^A )ce_A \otimes x \otimes\widetilde{de_A}
$$
for any $c, d\in C$, $x\in X$. In order to show that $f_1 \circ f =g$, we need to prove that
$$
\phi_{Y_1} (c\otimes x\otimes\widetilde{d})
=\theta_{C\otimes_A X \otimes_A \widetilde{C}}(c\otimes x\otimes\widetilde{d})
$$
for any $c, d\in C$, $x\in X$. For any $c, d\in C$, $x\in X$,
\begin{align*}
& \phi_{Y_1}(c\otimes x\otimes\widetilde{d}) \\
& =\Ind_W (E^A )^2 \sum_{i, j}u_i e_A \otimes E^Y (e_A u_i^* \cdot c\otimes x
\otimes\widetilde{d}\cdot u_j e_A )\otimes\widetilde{u_j e_A } \\
& =\Ind_W (E^A )^2 \sum_{i, j}u_i e_A \otimes E^Y (E^A (u_i^* c)\otimes x\otimes \widetilde{E^A (u_j^* d)})
\otimes\widetilde{u_j e_A } \\
& =\Ind_W (E^A )\sum_{i, j}u_i e_A \otimes E^A (u_i^* c)\cdot x\cdot E^A (d^* u_j )
\otimes\widetilde{u_j e_A } \\
& =\Ind_W (E^A )\sum_{i, j}u_i E^A (u_i^* c)e_A \otimes x\otimes [u_j E^A (u_j^* d)e_A ]^{\widetilde{}} \\
& =\Ind_W (E^A )ce_A \otimes x\otimes\widetilde{de_A} \\
& =\theta_{C\otimes_A X\otimes_A \widetilde{C}}(c\otimes x\otimes \widetilde{d})
\end{align*}
by the definition of $E^Y$ (See Preliminaries).
Therefore, we obtain the conclusion.
\end{proof}

By Lemmas \ref{lem:iso1}, \ref{lem:injective}, we can see that $(g^{-1}\circ f_1 )\circ f=\id$ on $\Pic(A, C)$.
Next, we shall show that
$$
f\circ (g^{-1}\circ f_1 )=\id
$$
on $\Pic(C, C_1 )$. Let $(Y, Y_1 )\in\Equi( C, C_1 )$. Then
$$
(g^{-1}\circ f_1 )([Y, Y_1 ])=g^{-1}([Y_1 , Y_2 ])=[\widetilde{C}\otimes_{C_1}Y_1 \otimes_{C_1} C \, , \,
\widetilde{C_1}\otimes_{C_2}Y_2 \otimes_{C_2}C_1 ] ,
$$
where $Y_2$ is the upward basic construction of $Y_1$ for $E^Y$. Thus
$Y_2 =C_1 \otimes_C Y\otimes_C \widetilde{C_1}$.
Hence
\begin{align*}
& (f\circ g^{-1}\circ f_1 )([Y, Y_1 ]) \\
&= [\widetilde{C_1}\otimes_{C_2}Y_2 \otimes_{C_2}C_1 \, , \,
C\otimes_A \widetilde{C}\otimes_{C_1} Y_1 \otimes_{C_1}C\otimes_A \widetilde{C} ] \\
& =[\widetilde{C_1}\otimes_{C_2}C_1\otimes_C Y\otimes_C \widetilde{C_1}\otimes_{C_2}C_1 \, , \,
C\otimes_A \widetilde{C}\otimes_{C_1} Y_1 \otimes_{C_1}C\otimes_A \widetilde{C} ] .
\end{align*}
By easy computations, $\widetilde{C_1}\otimes_{C_2}C_1\otimes_C Y\otimes_C \widetilde{C_1}\otimes_{C_2}C_1$
is isomorphic to $Y$ as $C-C$-equivalence bimodules by the linear map $\Phi_Y$ from 
$\widetilde{C_1}\otimes_{C_2}C_1\otimes_C Y\otimes_C \widetilde{C_1}\otimes_{C_2}C_1$
to $Y$ defined by
$$
\Phi_Y (\widetilde{c_1}\otimes y_2 \otimes d_1 )=E^C (c_1^* a_1 )\cdot y \cdot E^C (b_1^* d_1 )
$$
for any $a_1 , b_1 , c_1 , d_1 \in C_1$, $y\in Y$ and $y_2 =a_1 \otimes y\otimes \widetilde{b_1}$.
Also, 
$C\otimes_A \widetilde{C}\otimes_{C_1} Y_1 \otimes_{C_1}C\otimes_A \widetilde{C}$
is isomorphic to $Y_1$ as $C_1 -C_1$-equivalence bimodules by the linear map $\Phi_{Y_1}$ from
$C\otimes_A \widetilde{C}\otimes_{C_1} Y_1 \otimes_{C_1}C\otimes_A \widetilde{C}$
to $Y_1$ defined by
$$
\Phi_{Y_1}(c\otimes \widetilde{a}\otimes y_1 \otimes b\otimes\widetilde{d})
=ce_A a^* \cdot y_1 \cdot be_A d^*
$$
for any $a, b, c, d\in C$, $y_1 \in Y_1$. We note that
$\widetilde{C_1}\otimes_{C_2}C_1\otimes_C Y\otimes_C \widetilde{C_1}\otimes_{C_2}C_1$
is regarded as a closed subspace of
$C\otimes_A \widetilde{C}\otimes_{C_1} Y_1 \otimes_{C_1}C\otimes_A \widetilde{C}$ by the
linear map $\phi_{\widetilde{C_1}\otimes_{C_2}C_1\otimes_C Y\otimes_C \widetilde{C_1}\otimes_{C_2}C_1}$
from $\widetilde{C_1}\otimes_{C_2}C_1\otimes_C Y\otimes_C \widetilde{C_1}\otimes_{C_2}C_1$
to $C\otimes_A \widetilde{C}\otimes_{C_1} Y_1 \otimes_{C_1}C\otimes_A \widetilde{C}$
defined by
\begin{align*}
& \phi_{\widetilde{C_1}\otimes_{C_2}C_1\otimes_C Y\otimes_C \widetilde{C_1}\otimes_{C_2}C_1}
(\widetilde{c_1}\otimes a_1 \otimes y\otimes \widetilde{b_1}\otimes d_1 ) \\
& =\sum_{i, j}u_i \otimes E^{\widetilde{C}\otimes_{C_1} Y_1 \otimes_{C_1}C}
(u_i^* \cdot \widetilde{c_1}\otimes a_1 \otimes y\otimes\widetilde{b_1}\otimes d_1 \cdot u_j )\otimes \widetilde{u_j}
\end{align*}
for any $a_1 , b_1 , c_1 , d_1 \in C_1$, $y\in Y$, where $E^{\widetilde{C}\otimes_{C_1} Y_1 \otimes_{C_1}C}$
is the unique conditional expectation from
$\widetilde{C_1}\otimes_{C_2}C_1\otimes_C Y\otimes_C \widetilde{C_1}\otimes_{C_2}C_1$ onto
$\widetilde{C}\otimes_{C_1} Y_1 \otimes_{C_1}C$ with respect to $E^A$ and $E^A$, which
satisfies Conditions (1)-(6) in Definition \ref {def:pre6} by Lemma \ref{lem:expectation}.

\begin{lemma}\label{lem:expectation2}With the above notation,
$$
E^{\widetilde{C}\otimes_{C_1} Y_1 \otimes_{C_1}C}
(\widetilde{c_1}\otimes a_1 \otimes y \otimes \widetilde{b_1}\otimes d_1 )
=\widetilde{E^C (a_1^* c_1)}\otimes y\otimes E^C (b_1^* d_1 )
$$
for any $a_1 , b_1 , c_1 , d_1 \in C_1$, $y\in Y$.
\end{lemma}
\begin{proof}Let $y_2 =a_1 \otimes y\otimes\widetilde{b_1}$. Let $\{(r_i , r_i^* )\}$
be the quasi-basis for $E^{C_1}$, the dual conditional expectation of $E^C$ from $C_2$
onto $C_1$, which is defined by $r_i =\Ind_W (E^A )^{\frac{1}{2}}w_i e_C$,
where $e_C$ is the Jones projection for $E^C$. Let $F^C$ be the conditional expectation from
$C_1$, the $C_2 -C$-equivalence bimodule onto $C$, the $C_1 -A$-equivalence bimodule with
respect to $E^{C_1}$ and $E^A$, which is defined by
$$
F^C (ce_C d)=cE^C (d)e_C
$$
for any $c, d\in C$. Also, let $F^{\widetilde{C}}$ be the conditional expectation from $\widetilde{C_1}$
onto $\widetilde{C}$ induced by $F^C$. Then by Lemmas \ref{lem:conjugate}, \ref{lem:prepare3}, and
the definition of $E^{Y_1}$ (See Preliminaries),
\begin{align*}
& E^{\widetilde{C}\otimes_{C_1}Y_1 \otimes _{C_1}C}(\widetilde{c_1}\otimes y_2 \otimes d_1 ) \\
& =\sum_{i, j}F^{\widetilde{C}}(\widetilde{c_1}\cdot r_i )\otimes E^{Y_1}(r_i^* \cdot y_2 \cdot r_j )
\otimes F^C (r_j^* \cdot d_1 ) \\
& =\Ind_W (E^A )^2 \sum_{i, j}F^{\widetilde{C}}(\widetilde{E^C (w_i^* c_1 )})\otimes E^{Y_1}(E^C (w_i^* a_1 )
\otimes y \otimes \widetilde{E^C (w_j^* b_1 )}) \\
& \otimes F^C (E^C (w_j^* d_1 )) \\
& =\Ind_W (E^A )\sum_{i, j}[F^C (E^C (w_i^* c_1 ))]^{\widetilde{}}\otimes E^C (w_i^* a_1 )\cdot y \cdot
E^C (w_j^* b_1 )^* \\
& \otimes F^C (E^C (w_j^* d_1 )) \\
& =\Ind_W (E^A)\sum_{i, j}[F^C (E^C (a_1^* w_i E^C (w_i^* c_1 )))]^{\widetilde{}}\otimes y \\
& \otimes F^C (E^C (E^C (b_1^* w_j )w_j^* d_1 ))\\
& =\Ind_W (E^A )[F^C (E^C (a_1^* c_1 ))]^{\widetilde{}}\otimes y\otimes F^C (E^C (b_1^* d_1 )) .
\end{align*}
Here since we regard $E^C (a_1^* c_1 )$ as an element in the $C_2 -C$-equivalence bimodule $C_1$,
\begin{align*}
F^C (E^C (a_1^* c_1 )) & =\sum_i F^C (E^C (a_1^* c_1 )w_i e_C w_i^* ) \\
& =\sum_i E^C (a_1^* c_1 )w_i E^C (w_i^* )e_C =E^C (a_1^* c_1 )e_C .
\end{align*}
Since we regard the element $E^C (a_1^* c_1 )e_C$ in the $C_2 -C$-equivalence bimodule $C_1$
as the element $\Ind_W (E^A )^{-\frac{1}{2}}E^C (a_1^* c_1 )$ in $C$, the $C_1 -A$-equivalence bimodule
by \cite [Section 4]{KT4:morita}, 
$$
F^C (E^C (a_1^* c_1 ))=\Ind_W (E^A )^{-\frac{1}{2}}E^C (a_1^* c_1 ) .
$$
Similarly, $F^C (E^C (b_1^* d_1 ))=\Ind_W (A)^{-\frac{1}{2}}E^C (b_1^* d_1 )$. Therefore, we obtain the
conclusion.
\end{proof}

\begin{lemma}\label{lem:surjective}With the above notation,
$f\circ g^{-1}\circ f_1 =\id$ on $\Pic(C, C_1 )$.
\end{lemma}
\begin{proof}Let $\Phi_Y$ be the isomorphism of
$\widetilde{C_1}\otimes_{C_2}C_1 \otimes_C Y\otimes_C  \widetilde{C_1}\otimes_{C_2} C_1$
onto $Y$ and $\Phi_{Y_1}$ the isomorphism of
$C\otimes_A \widetilde{C}\otimes_{C_1} Y_1 \otimes _{C_1}C\otimes_A \widetilde{C}$
onto $Y_1$ defined in the above. Also, let $\phi_{\widetilde{C_1}\otimes_{C_2}C_1 \otimes_C Y \otimes_C
\widetilde{C_1}\otimes_{C_2}C_1}$ be the injective linear map from
$\widetilde{C_1}\otimes_{C_2}C_1 \otimes_C Y\otimes_C  \widetilde{C_1}\otimes_{C_2} C_1$ into
$C\otimes_A \widetilde{C}\otimes_{C_1} Y_1 \otimes _{C_1}C\otimes_A \widetilde{C}$
defined in the above. It suffices to show that
$$
\Phi_Y =\Phi_{Y_1} \circ \phi_{\widetilde{C_1}\otimes_{C_2}C_1 \otimes_C Y \otimes_C
\widetilde{C_1}\otimes_{C_2}C_1}
$$
in order to prove that $f\circ g^{-1}\circ f_1 =\id$ on $\Pic(C, C_1 )$.
Let $a_1 , b_1 , c_1 , d_1 \in C_1$, $y\in Y$ and let $y_2 =a_1 \otimes y \otimes \widetilde{b_1}$.
Then by Lemma \ref{lem:expectation2},
\begin{align*}
& \phi_{\widetilde{C_1}\otimes_{C_2}C_1 \otimes_C Y \otimes_C
\widetilde{C_1}\otimes_{C_2}C_1}(\widetilde{c_1}\otimes y_2 \otimes d_ 1 ) \\
& =\sum_{i, j}u_i \otimes E^{\widetilde{C}\otimes_{C_1}Y_1 \otimes _{C_1} C}
(u_i^* \cdot \widetilde{c_1}\otimes a_1 \otimes y\otimes\widetilde{b_1}\otimes d_1 \cdot u_j )\otimes\widetilde{u_j} \\
& =\sum_{i, j}u_i \otimes E^{\widetilde{C}\otimes_{C_1}Y_1 \otimes _{C_1} C}
(\widetilde{c_1 u_i}\otimes a_1 \otimes y\otimes\widetilde{b_1}\otimes d_1 u_j )\otimes \widetilde{u_j} \\
& =\sum_{i, j}u_i \otimes [E^C (a_1^* c_1 u_i )]^{\widetilde{}}
\otimes y\otimes E^C (b_1^* d_1 u_j )\otimes \widetilde{u_j} .
\end{align*}
Hence
\begin{align*}
& (\Phi_{Y_1}\circ \phi_{\widetilde{C_1}\otimes_{C_2}C_1 \otimes_C Y \otimes_C
\widetilde{C_1}\otimes_{C_2}C_1})(\widetilde{c_1}\otimes y_2 \otimes d_1 ) \\
& =\sum_{i, j}u_i e_A E^C (u_i^* c_1^* a_1 )\cdot y\cdot E^C (b_1^* d_1u_j )e_A u_j^* \\
& =\sum_{i, j}u_i e_A u_i^* E^C (c_1^* a_1 )\cdot y\cdot E^C (b_1^* d_1 )u_j e_A u_j^* \\
& =E^C (c_1^* a_1 )\cdot y\cdot E^C (b_1^* d_1 ) .
\end{align*}
On the other hand,
$$
\Phi_Y (\widetilde{c_1}\otimes y_2 \otimes d_1 )=E^C (c_1^* a_1 )\cdot y\cdot E^C (b_1^* d_1 ) .
$$
Therefore, we obtain the conclusion.
\end{proof}

\begin{thm}\label{thm:bijective}Let $A\subset C$ be a unital inclusion of unital $C^*$-algebras.
We suppose that $A' \cap C=\BC1$ and that there is a conditional expectation $E^A$ of
Watatani index-finite type from $C$ onto $A$. Let $C_1$ be the $C^*$-basic construction for $E^A$.
Then $\Pic(A, C)\cong \Pic(C, C_1 )$.
\end{thm}
\begin{proof}This is immediate by Lemmas \ref{lem:injective}, \ref{lem:surjective}.
\end{proof}

\section{The Picard groups for inclusions of $C^*$-algebras and the ordinary Picard groups}
\label{sec:relation}
In this section, we shall investigate the relation between the Picard groups
for inclusion of $C^*$-algebras and the ordinary Picard groups. Let $A\subset C$ be an inclusion of
$C^*$-algebras with $\overline{AC}=C$ and let $f_C$ be the homomorphism from $\Pic(A, C)$ to $\Pic(C)$
defined by
$$
f_C : \Pic(A, C)\rightarrow\Pic(C): [X, Y]\mapsto [Y] ,
$$
where $\Pic(C)$ is the ordinary Picard group of $C$.
In this section, we suppose that $\overline{AC}=C$ for any inclusions of $C^*$-algebras
$A\subset C$. 

Let $u$ be a unitary element in $M(C)$ satisfying that $uau^*$, $u^* a u\in A$ for any
$a\in A$. We regard $Au$ as an $A-A$-equivalence bimodule as follows:
In the usual way, we regard $Au$ as a vector space over $\BC$. We define
the left $A$-action and the right action by
$$
a\cdot xu=axu, \quad xu\cdot a =xua =x(uau^* )u
$$
for any $a, x\in A$. We also define the left $A$-valued inner product and
the right $A$-valued inner product by
$$
{}_A \la xu, yu \ra =xuu^* y^* =xy^* , \quad \la xu, yu \ra_A =u^* x^* yu
$$
for any $x, y\in A$. Furthermore, $Au$ is a closed subspace of $C$,
the trivial $C-C$-equivalence bimodule and we can
see that $[Au, C]$ is an element in $\Pic(A, C)$ by easy computations.
Let $\Aut (A, C)$ be the group of all automorphisms of $C$ such that $\alpha|_A$ is
an automorphism of $A$. Let $\alpha\in \Aut (A, C)$ and let $[X_{\alpha}, Y_{\alpha}]$
be the element in $\Pic(A, C)$ induced by $\alpha$, which is defined in Section \ref{sec:definition}

\begin{lemma}\label{lem:kernel}With the above notation, let $\alpha\in \Aut (A, C)$. Then the following conditions are
equivalent:
\newline
$(1)$ $[X_{\alpha}, Y_{\alpha}]\in\Ker f_C$,
\newline
$(2)$ There is a unitary element $u\in M(C)$ satisfying that
$uau^*$, $u^* au\in A$ for any $a\in A$ and that
$[X_{\alpha}, Y_{\alpha}]=[Au , C]$ in $\Pic(A, C)$.
\end{lemma}
\begin{proof}$(1)\Rightarrow (2)$: We suppose that $[X_{\alpha}, Y_{\alpha}]\in\Ker f_C$.
Then $[Y_{\alpha}]=[C]$ in $\Pic(C)$. Hence there is a unitary element $u\in M(C)$
such that $\alpha=\Ad(u)$ on $C$. Since $\alpha|_A$ is also an automorphism of $A$, we can
see that $uau^*$, $u^* au\in A$ for any $a\in A$. Let $\Phi$ be the linear map from $C$ to
$Y_{\Ad(u)}$ defined by $\Phi(x)=xu^*$ for any $x\in C$. Then by the proof of Lemma \ref{lem:exact1},
$\Phi$ is a $C-C$-equivalence bimodule isomorphism of $C$ onto $Y_{\Ad(u)}$.
Also, we can see that $\Phi|_{Au}$ is an $A-A$-equivalence bimodule isomorphism
of $Au$ onto $X_{\alpha}$. Indeed, for any $x\in A$,
$\Phi(xu)=xuu^* =x\in X_{\alpha}$. By Lemma \ref{lem:simple} $\Phi|_{Au}$ is an
$A-A$-equivalence bimodule isomorphism of $Au$ onto $X_{\alpha}$.
\newline
$(2)\Rightarrow (1)$: It is clear by the definition of $f_C$.
\end{proof}

Let $K(M(C))$ be the set of all unitary elements $u$ in $M(C)$ such that
$uau^*$, $u^* au\in A$ for any $a\in A$. Then $K(M(C))$ is a subgroup of
the group of all unitary elements in $M(C)$.

\begin{lemma}\label{lem:product2}With the above notation, for any $u, v\in K(M(C))$,
$$
[Au, C][Av, C]=[Auv, C]
$$
in $\Ker f_C$.
\end{lemma}
\begin{proof}By the definition of the product in $\Pic(A, C)$,
$$
[Au, C][Av, C]=[Au\otimes_A Av, C\otimes_C C]
$$
in $\Pic(A, C)$. Hence we have only to show that
$$
[Au\otimes_A Av, C\otimes_C C]=[Auv, C]
$$
in $\Pic(A, C)$. Let $\Phi$ be the linear map from $C\otimes_C C$ to $C$ defined by
$$
\Phi(x\otimes y)=xy
$$
for any $x, y\in C$. Then clearly $\Phi$ is a $C-C$-equivalence bimodule isomorphism
of $C\otimes_C C$ onto $C$. Also, for any $x, y\in A$,
$$
\Phi(xu\otimes yv)=xuyv=xuyu^* uv\in Auv .
$$
Hence by Lemma \ref{lem:simple}, we can see that $\Phi|_{Au\otimes_A Av}$ is an
$A-A$-equivalence bimodule isomorphism of $Au\otimes_A Av$ onto $Auv$.
Therefore, we obtain the conclusion.
\end{proof}

\begin{lemma}\label{lem:equivalence1}With the above notation, let $u\in K(M(C))$.
Then the following conditions are equivalent:
\newline
$(1)$ $[Au, C]=[A, C]$ in $\Pic(A, C)$,
\newline
$(2)$ There is a unitary element $w\in M(A)$ such that
$w^* u\in C' \cap M(C)$.
\end{lemma}
\begin{proof}$(1)\Rightarrow (2)$: By the proof of Lemma \ref{lem:kernel},
$[Au, C]=[X_{\Ad(u)}, Y_{\Ad(u)}]$ in $\Pic(A, C)$. Hence since
$[X_{\Ad(u)}, Y_{\Ad(u)}]=[A, C]$ in $\Pic(A, C)$, by the proof of Lemma \ref{lem:exact1},
there is a unitary element $w\in M(A)$ such that
$\Ad(u)=\Ad(w)$
on $C$. Hence $w^* u\in C' \cap M(C)$ since $M(A)\subset M(C)$.
\newline
$(2)\Rightarrow (1)$: Since $w^* u\in C' \cap M(C)$, $\Ad(u)=\Ad(w)$ on $C$.
Thus
$$
[Au, C]=[X_{\Ad(u)}, Y_{\Ad(u)}]=[X_{\Ad(w)}, Y_{\Ad(w)}]=[A, C]
$$
in $\Pic(A, C)$ by the proofs of Lemmas \ref{lem:exact1} and \ref{lem:kernel}.
\end{proof}

Let $U(M(A))$ be the group of all unitary elements in $M(A)$. Then $U(M(A))$ is a
subgroup of $K(M(C))$ since $M(A)\subset M(C)$ is a unital inclusion.

\begin{lemma}\label{lem:normal}With the above notation, $U(M(A))$ is a normal
subgroup of $K(M(C))$.
\end{lemma}
\begin{proof}Let $u\in K(M(C))$ and $w\in U(M(A))$. Let $\{w_{\lambda}\}_{\lambda\in\Lambda}$ be
a net in $A$ such that $\{w_{\lambda}\}_{\lambda\in\Lambda}$ is strictly convergent
to $w\in M(A)$. Then since $uw_{\lambda}u^*\in A$ for any $\lambda\in\Lambda$,
$uwu^* \in U(M(A))$. Therefore we obtain the conclusion.
\end{proof}

Let $\BK$ be the $C^*$-algebra of all compact operators on a countably inifinite
dimensional Hilbert space. Let $A\subset C$ be an inclusion of $C^*$-algebras with $\overline{AC}=C$
and $A^s \subset C^s$ the inclusion of $C^*$-algebras induced by $A\subset C$, where
$A^s =A\otimes \BK$ and $C^s =C\otimes \BK$. Let $S(A^s , C^s )$ be the subgroup of $\Pic(C^s )$
defined by
$$
S(A^s , C^s )=\{ [Y_{\alpha}]\in \Pic(C^s ) \, | \, \alpha\in \Aut (A^s , C^s ) \} .
$$
Then by Proposition \ref {prop:exact2}, $S(A^s, C^s )=\Ima f_{C^s}$, where $f_{C^s}$ is
the homomorphism of $\Pic(A^s, C^s )$ to
$\Pic(C^s )$ defined in the same way as in the above. By Lemma \ref{lem:kernel}
$$
\Ker f_{C^s}= \{[A^s u \, , C^s ] \in \Pic(A^s , C^s ) \, | \, u\in K(M(C^s )) \}.
$$ 
Furthermore, we suppose that $(C^s )' \cap M(C^s )=\BC1$. Then by Lemmas \ref{lem:product2}, \ref{lem:equivalence1},
$\Ker f_{C^s}\cong K(M(C^s ))/U(M(A^s ))$
as groups. Thus, we obtain the following theorem:

\begin{thm}\label{thm:exact3}Let $A\subset C$ be an inclusion of $C^*$-algebras such that
$A$ is $\sigma$-unital and $\overline{AC}=C$. Let
$A^s \subset C^s$ be the inclusion of $C^*$-algebras induced by $A\subset C$,
where $A^s =A\otimes\BK$ and $C^s =C\otimes \BK$. We suppose that
$(C^s)' \cap M(C^s )=\BC1_{M(C^s)}$. Then we have the following exact sequence:
$$
1\longrightarrow K(M(C^s ))/U(M(A^s ))\longrightarrow\Pic(A^s , C^s )\overset{f_C^s }
\longrightarrow S(A^s , C^s )\longrightarrow 1 .
$$
\end{thm}

Again, we consider an inclusion of $C^*$-algebras, $A\subset C$ with $\overline{AC}=C$.
Let $f_A$ be the homomorphism of $\Pic(A, C)$ to $\Pic(A)$ defined by
$$
f_A : \Pic(A, C)\to \Pic(A): [X, Y]\mapsto [X] ,
$$
where $\Pic(A)$ is the ordinary Picard group of $A$. Let $\Aut_0 (A, C)$ be the group of all
automorphisms $\alpha$ of $C$ with $\alpha=\id$ on $A$. Then by easy computations, $\Aut_0 (A, C)$
is a normal subgroup of $\Aut (A, C)$.

\begin{lemma}\label{lem:kernel2}With the above notation, let $\alpha\in \Aut(A, C)$. Then the following
conditions are equivalent:
\newline
$(1)$ $[X_{\alpha}, Y_{\alpha}]\in\Ker f_A$,
\newline
$(2)$ There is a $\beta\in Aut_0 (A, C)$ such that $[X_{\alpha}, Y_{\alpha}]=[X_{\beta}, Y_{\beta}]$
in $\Pic(A, C)$.
\end{lemma}
\begin{proof}$(1)\Rightarrow (2)$: Since $[X_{\alpha}, Y_{\alpha}]\in\Ker f_A$, $[X_{\alpha}]=[A]$ in
$\Pic(A)$. Hence there is a unitary element $u\in M(A)$ such that $\alpha =\Ad(u)$ on
$A$. Since $[X_{\Ad(u^* )}, Y_{\Ad(u^* )}]=[A, C]$ in $\Pic(A, C)$,
$$
[X_{\alpha} \, , Y_{\alpha}]=[X_{\Ad(u^* )} \, , Y_{\Ad(u^* )}][X_{\alpha} \, , Y_{\alpha}]
=[X_{\Ad(u^* )\circ\alpha} \, , Y_{\Ad(u^* )\circ\alpha}]
$$
in $\Pic(A)$. Let $\beta=\Ad(u^* )\circ\alpha$. Then $\beta\in \Aut_0 (A, C)$.
\newline
$(2)\Rightarrow (1)$: Since $X_{\beta}=A$, $[X_{\beta} \, , Y_{\beta}]\in\Ker f_A$.
Hence $[X_{\alpha} , \, Y_{\alpha}]\in\Ker f_A$.
\end{proof}

Let $\pi$ be the homomorphism of $\Aut(A, C)$ to $\Pic(A, C)$
defined by
$$
\pi(\alpha)=[X_{\alpha} , \, Y_{\alpha}]
$$
for any $\alpha\in Aut(A, C)$. Let $\Aut_I (A, C)$ be the subset of $\Aut (A, C)$ defined by
$$
\Aut_I (A, C)=\{\alpha\in\Aut(A, C) \, | \, \alpha|_A \in \Int(A) \} .
$$
Then clearly $\Aut_I (A, C)$ is a subgroup of $\Aut(A, C)$. Also, $\Aut_I (A, C)$
is a normal subgroup of $\Aut (A, C)$. Indeed, let $\alpha\in\Aut_I (A, C)$. Then
there is a unitary element $u\in M(A)$ such that $\alpha(a)=uau^*$ for any $a\in A$.
Hence for any $\beta\in \Aut(A, C)$ and $a\in A$,
$$
(\beta\circ\alpha\circ\beta^{-1})(a)=\beta(u\beta^{-1}(a)u^* )=\underline{\beta}(u)a\underline{\beta}(u^* )
$$
since $\beta|_A \in \Aut(A)$, where $\underline{\beta}$ is an automorphism of $M(C)$ induced by $\beta$,
whose restriction $\underline{\beta}|_{M(A)}$ is also an automorphism of $M(A)$.
Thus $\Aut(A, C)$ is a normal subgroup of $\Aut(A, C)$.

\begin{lemma}\label{lem:kernel3}With the above notation, let $\alpha\in \Aut (A, C)$. Then the following
conditions are equivalent:
\newline
$(1)$ $[X_{\alpha} \, , Y_{\alpha}]\in \Ker f_A$,
\newline
$(2)$ $\alpha\in\Aut_I (A, C)$.
\end{lemma}
\begin{proof}$(1)\Rightarrow (2)$: By Lemma \ref{lem:kernel2}, there is a $\beta\in\Aut_0 (A, C)$ such that
$[X_{\alpha} \, , Y_{\alpha}]=[X_{\beta} \, , Y_{\beta}]$ in $\Pic(A, C)$. Hence by Lemma \ref{lem:exact1},
$\alpha\circ\beta^{-1}\in\Int(A, C)$. Thus there is a unitary element $u\in M(A)$ such that
$\alpha\circ\beta^{-1}=\Ad(u)$, that is, $\alpha=\Ad(u)\circ\beta$. Then for any $a\in A$,
$\alpha(a)=u\beta(a)u^* =uau^*$. Hence $\alpha\in\Aut_I (A, C)$.
\newline
$(2)\Rightarrow (1)$: Since $\alpha\in\Aut_I (A, C)$, there is a unitary element
$u\in M(A)$ such that $\alpha(a)=uau^*$ for any $a\in A$. Let $\beta=\Ad(u^* )\circ\alpha$. Then
$\beta\in\Aut(A, C)$ since $M(A)\subset M(C)$ is a unital inclusion. Also, $\beta(a)=a$ for any $a\in A$.
Hence $\beta\in \Aut_0 (A, C)$. Thus $[X_{\alpha} , \, Y_{\alpha}]=[X_{\beta}, \, Y_{\beta}]$
in $\Pic(A, C)$ by Lemma \ref{lem:exact1}. Hence by Lemma \ref{lem:kernel2},
$[X_{\alpha} , \, Y_{\alpha}]\in\Ker f_A$.
\end{proof}

\begin{lemma}\label{lem:kernel4}With the above notations, let  $\alpha\in\Aut(A, C)$.
Then the following conditions are equivalent:
\newline
$(1)$ $[X_{\alpha}]=[A]$ in $\Pic(A)$,
\newline
$(2)$ $\alpha\in\Aut_I (A, C)$.
\end{lemma}
\begin{proof}$(1)\Rightarrow (2)$: By \cite [Proposition 3.1]{BGR:linking}, there is
a unitary element $u\in M(A)$ such that $\alpha|_A =\Ad(u)$. Hence $\alpha\in \Aut_I (A, C)$.
\newline
$(2)\Rightarrow (1)$: Since $\alpha|_A \in\Int (A)$, we can see that $[X_{\alpha}]=[A]$ in $\Pic(A)$.
\end{proof}

Let $A^s \subset C^s$ be the inclusion of $C^*$-algebras induced by the inclusion of
$C^*$-algebras $A\subset C$ with $\overline{AC}=C$. We suppose that $A$ is
$\sigma$-unital. Then by Proposition \ref{prop:exact2},
the homomorphism $\pi$ of $\Aut(A^s , C^s )$to $\Pic(A^s , C^s )$ is surjective. Hence
$$
\Ima f_{A^s}=\{[X_{\alpha}]\in\Pic(A^s ) \, | \, \alpha\in \Aut(A^s , C^s ) \}
$$
and by Lemma \ref {lem:kernel2},
$$
\Ker f_{A^s}=\{[X_{\alpha}, \, Y_{\alpha}]\in\Pic(A^s , C^s ) \, | \, \alpha\in\Aut_0 (A^s , C^s ) \} .
$$
Also, we have the exact sequence
$$
1\longrightarrow \Ker f_{A^s} \longrightarrow \Pic(A^s , C^s )\longrightarrow \Ima f_{A^s} \longrightarrow 1 .
$$
Since $\pi$ is a surjective homomorphism of $\Aut_I (A^s , C^s )$ onto $\Ker f_{A^s}$ by
Lemma \ref{lem:kernel3}, we can see that
$$
\Ker f_{A^s}\cong \Aut_I (A^s , C^s )/\Int (A^s , C^s )
$$
by Proposition \ref{prop:exact2}. Also, $f_{A^s}\circ\pi$ is a surjective homomorphism of $\Aut(A^s , C^s )$
onto $\Ima f_{A^s}$, we can see that
$$
\Ima f_{A^s} \cong \Aut (A^s , C^s )/\Aut_I (A^s , C^s )
$$
by Lemma \ref{lem:kernel4}. By the above discussions, we obtain the
following theorem:

\begin{thm}\label{thm:exact4}Let $A\subset C$ be an inclusion of
$C^*$-algebras such that $A$ is $\sigma$-unital and $\overline{AC}=A$.
Let $A^s \subset C^s$ be the inclusion of $C^*$-algebras induced by
$A\subset C$, where $A^s =A\otimes \BK$ and $C^s =C\otimes \BK$. Then
we have the following:
\begin{align*}
1\longrightarrow & \Ker f_{A^s}  \longrightarrow \Pic(A^s , C^s )\longrightarrow \Ima f_{A^s} \longrightarrow 1 , \\
& \Ker f_{A^s}\cong \Aut_I (A^s , C^s )/\Int (A^s , C^s ) , \\
& \Ima f_{A^s} \cong \Aut (A^s , C^s )/\Aut_I (A^s , C^s ) .
\end{align*}
\end{thm}

\section{Examples}\label{sec:example} In this section, we shall give some examples
of the Picard groups for unital inclusions of unital $C^*$-algebras obtained by
routine computations. Some other examples can be found in \cite [Corollary 6.7]{Kodaka:generalized}.
We begin this section with a trivial
example.

\begin{exam}\label{exam:trivial} Let $A$ be a $C^*$-algebra.
We regard $A\subset A$ as an inclusion of $C^*$-algebras by the
identity map. Then $\Pic (A, A)\cong \Pic(A)$.
\end{exam}
\begin{proof} Let $\jmath$ be the map from $\Pic(A)$ to $\Pic(A, A)$ defined by
$$
\jmath([X])=[X, X]
$$
for any $[X]\in\Pic(A)$. Then clearly $\jmath$ is a monomorphism of $\Pic(A)$ to $\Pic(A, A)$.
Let $[X, Y]\in \Pic(A, A)$. Since $X$ is an $A-A$-equivalence bimodule,
$\overline{A\cdot X}=X=\overline{X\cdot A}$
by \cite [Proposition1.7]{BMS:quasi}.
Also, since $Y$ is an $A-A$-equivalence bimodule, by \cite [Proposition 1.7]{BMS:quasi},
$$
Y=\overline{A\cdot Y}=\overline{{}_A \la X\, , \, X \ra\cdot Y}=\overline{X\cdot \la X \, , Y \ra_A}
=\overline{X\cdot A}=X .
$$
Thus
$\jmath$ is surjective. Therefore, $\jmath$ is an isomorphism of $\Pic(A)$ onto $\Pic(A, A)$.
\end{proof}

Let $A$ be a unital $C^*$-algebra and we consider the unital inclusion of unital
$C^*$-algebras $\BC1 \subset A$.
Before we give the next example, we prepare a lemma.

\begin{lemma}\label{lem:kernel5} Let $A$ be a unital $C^*$-algebra satisfying the following
sequence
$$
1\longrightarrow\Int(A)\longrightarrow\Aut (A)\longrightarrow\Pic (A)\longrightarrow 1
$$
is exact. Then we have the following exact sequence:
$$
1\longrightarrow \Ker f_A \longrightarrow \Pic (\BC1, A)\overset{f_A}\longrightarrow\Pic (A)\longrightarrow 1.
$$
Furthermore,
$$
\Ker f_A =\{\, [\BC u, A ]\in \Pic (\BC1 , A) \, | \, u\in U(A) \} ,
$$
where $U(A)$ is the group of all unitary elements in $A$.
\end{lemma}
\begin{proof}
By the assumptions, for any $[X]\in\Pic (A)$, there is an automorphism $\alpha$ of $A$ such that
$[X]=[X_{\alpha}]$ in $\Pic(A)$, where $X_{\alpha}$ is the $A-A$-equivalence bimodule induced by
$\alpha$. Then $\BC 1$ is a closed subspace of $X_{\alpha}$ and by easy computations, we can see that
$[\BC 1 , X_{\alpha}]\in\Pic (\BC1, A)$ and that $f_A ([\BC 1, X_{\alpha}])=[X_{\alpha}]$. Thus
$f_A$ is surjective. Next, we show that
$$
\Ker f_A =\{\, [\BC u, A ]\in \Pic (\BC1 , A) \, | \, u\in U(A) \} .
$$
Let $[X, Y]\in \Ker f_A$. Then $[Y]=[A]$ in $\Pic(A)$. Hence there is an
$A-A$-equivalence bimodule isomorphism $\theta$ of $Y$ onto $A$. Let $Z=\theta(X)$.
Then $[Z, A]\in \Pic(\BC1, A)$ and $[Z, A]=[X, Y]$ in $\Pic(\BC 1, A)$ by easy computations.
Also, since $\Pic(\BC)=\{[\BC]\}$, $Z$ is a vector space over $\BC$ of dimension 1. Thus there is
a unitary element $u$ in $A$ with $||u||=1$ such that $Z=\BC u$. We note that
$$
{}_{\BC} \la u, u \ra ={}_A \la u, u \ra=uu^* \in \BC , \quad
\la u, u \ra_{\BC}=\la u, u \ra_A =u^* u \in \BC .
$$
Since $uu^*$, $u^* u$ are positive numbers and $||u||=1$, $u$ is a unitary element in $A$.
Thus
$$
\Ker f_A \subset \{\, [\BC u, A ]\in \Pic (\BC1 , A) \, | \, u\in U(A) \} .
$$
Furthermore, by Lemma \ref{lem:kernel} we can see that $[\BC u, A]\in \Ker f_A$ for any $u\in U(A)$ . Hence
$$
\Ker f_A = \{\, [\BC u, A ]\in \Pic (\BC1 , A) \, | \, u\in U(A) \} .
$$
\end{proof}

Let $\jmath$ be the map from $\Pic (A)$ to $\Pic(\BC 1, A)$ defined by
$$
\jmath([X_{\alpha}])=[\BC 1, X_{\alpha}]
$$
for any $\alpha \in \Aut (A)$. Then by the proof of Lemma \ref{lem:kernel5}, $\jmath$ is
a homomorphism of $\Pic (A)$ to $\Pic(\BC 1, A)$ with $f_A \circ\jmath=\id$ on $\Pic (A)$.
Also, let $\kappa$ be the map from $U(A)$ to $\Ker f_A$ defined by
$$
\kappa(u)=[\BC u, A]
$$
for any $u\in U(A)$. Then by Lemma \ref{lem:equivalence1}, $\Ker \kappa=U(A\cap A' )$
and $\kappa$ is surjective. Hence we obtain the following example.

\begin{exam}\label{exam:scalar} Let $A$ be a unital $C^*$-algebra satisfying the following
sequence
$$
1\longrightarrow\Int(A)\longrightarrow\Aut (A)\longrightarrow\Pic (A)\longrightarrow 1
$$
is exact. Then $\Pic(\BC 1, A)$ is isomorphic to a semidirect product group of $U(A)/U(A\cap A' )$ by
$\Pic(A)$, that is,
$$
\Pic(\BC 1, A)\cong U(A)/U(A\cap A' )\rtimes_s \Pic(A) .
$$
\end{exam}

We go to the next example. Let $A$ be a unital $C^*$-algebra satisfying the following
sequence
$$
1\longrightarrow\Int(A)\longrightarrow\Aut (A)\longrightarrow\Pic (A)\longrightarrow 1
$$
is exact. Let $n$ be any positive integer with $n\geq 2$.
We consider the unital inclusion of unital $C^*$-algebras
$$
a\in A \mapsto a\otimes I_n \in M_n (A) ,
$$
where $I_n$ is the unit element in $M_n (A)$.

\begin{lemma}\label{lem:surjective2} With the above notation, the homomorphism $f_A$
of
\newline
$\Pic(A, M_n (A))$ to $\Pic(A)$ is surjective.
\end{lemma}
\begin{proof} Let $X$ be any $A-A$-equivalence bimodule.
Then by the assumptions, there is an automorphism $\alpha$ of $A$ such that
$[X]=[X_{\alpha}]$ in $\Pic(A)$, where $X_{\alpha}$ is the
$A-A$-equivalence bimodule induced by $\alpha$. Let
$X_{\alpha\otimes\id}$ be the $M_n (A)-M_n (A)$-equivalence bimodule induced by $\alpha\otimes\id$. Then
$[X_{\alpha}, X_{\alpha\otimes\id}]\in \Pic(A , M_n (A))$ and
$$
f_{A}([X_{\alpha}, X_{\alpha\otimes\id}])=[X_{\alpha}]=[X]
$$
in $\Pic(A)$. Thus $f_{A}$ is surjective.
\end{proof}

Let $\jmath$ be the map from $\Pic(A)$ to
$\Pic(A, M_n (A))$ defined by
$$
\jmath ([X_{\alpha}])=[X_{\alpha}, X_{\alpha\otimes\id}]
$$
for any $\alpha\in \Aut(A)$. Then $\jmath$ is clearly a homomorphism of $\Pic(A)$ to
$\Pic(A, M_n (A))$ with $f_A \circ\jmath =\id$ on $\Pic(A)$.
Thus $\Pic(A, M_n (A))$ is a semidirect  group of $\Ker f_A$ by $\Pic(A)$ by Lemma \ref{lem:surjective2}.

We compute $\Ker f_A$. Let $[X, Y]\in \Ker f_{A}$. Since $f_A ([X, Y])=[X]$ in $\Pic (A)$,
there is an $A-A$-equivalence bimodule isomorphism $\theta$
of $A$ onto $X$. We prepare the following lemma.

\begin{lemma}\label{lem:unit}Let $A\subset C$ be a unital inclusion of unital $C^*$-algebras.
Let $(X, Y)\in\Equi (A, C)$ satisfying that there is an $A-A$-equivalence bimodule isomorphism
$\theta$ of $A$ onto $X$. Then there is a $\beta \in \Aut_0 (A, C)$ such that
$$
[X, Y]=[A, Y_{\beta}]
$$ in $\Pic(A, C)$, where $Y_{\beta}$ is the $C-C$-equivalence bimodule induced by $\beta$.
\end{lemma}
\begin{proof} By the discussions in \cite [Section 2]{KT4:morita} and Rieffel \cite [Proposition 2.1]{Rieffel:rotation},
there is an automorphism $\beta$ of $C$ defined by
$$
\beta(c)={}_C \la \theta(1)\cdot c \, , \, \theta(1) \ra
$$
for any $c\in C$ since
$$
{}_C \la \theta(1 ) \, , \, \theta(1 ) \ra={}_A \la \theta(1) \, , \, \theta(1) \ra ={}_A \la 1 \, , \ 1 \ra=1
$$
and also we have $\la \theta(1) \, , \, \theta(1)\ra_C =1$, where $1$ is the unit element of $A$.
Then for any $a\in A$
$$
\beta (a)={}_C \la \theta(1)\cdot a \, , \, \theta(1) \ra ={}_C \la \theta (a) \, , \, \theta(1) \ra
={}_A \la \theta (a) , \theta (1) \ra ={}_A \la a, 1 \ra =a .
$$
Hence $\beta\in \Aut_0 (A, C)$. Let $\rho$ be the linear map from $Y$ to $Y_{\beta}$ defined by
$$
\rho(y)={}_C \la y \, , \, \theta(1) \ra 
$$
for any $y\in Y$. For any $c\in C$, $c\cdot \theta(1)\in Y$ and
$$
\rho (c\cdot \theta(1)) ={}_C \la c\cdot \theta(1) \, , \, \theta(1) \ra =c\, {}_C \la \theta(1) \, , \, \theta(1) \ra 
=c .
$$
Hence $\rho$ is surjective. Also, for any $y, z\in Y$,
\begin{align*}
{}_C \la \rho(y) \, , \, \rho(z) \ra& ={}_C \la \, {}_C \la y \, , \, \theta(1) \ra \, , \, {}_C \la z \, , \, \theta(1) \ra \ra
={}_C \la y \, , \, \theta(1) \ra \, {}_C \la \theta(1) \, , \, z \ra \\
& ={}_C \la \, {}_C \la y \, , \, \theta(1) \ra \cdot \theta(1) \, , \, z \ra
={}_C \la y\cdot \la \theta(1) \, , \, \theta(1) \ra_C \, , \, z \ra \\
& ={}_C \la y\cdot \la 1 \, , \, 1 \ra_A \, , \, z \ra ={}_C \la y \, , \, z \ra ,
\end{align*}
\begin{align*}
\la \rho(y) \, , \, \rho(z) \ra_C & =\la\, {}_C \la y \, , \, \theta(1) \ra \, , \, {}_C \la z \, , \, \theta(1) \ra \ra_C \\
& =\beta^{-1}( \, {}_C \la \theta(1) \, , \, y \ra \, {}_C \la z \, , \, \theta(1) \ra ) \\
& =\beta^{-1}( \, {}_C \la \, {}_ C \la \theta(1) \, , \, y \ra \cdot z \, , \, \theta(1) \ra) \\
& =\beta^{-1}(\, {}_C \la \theta(1)\cdot \la y \, , \, z \ra_C \, , \, \theta(1) \ra) \\
& =(\beta^{-1}\circ\beta)(\la y \, , \, z \ra_C )=\la y \, , \, z \ra_C .
\end{align*}
Hence
$\rho$ is a $C-C$-equivalence bimodule isomorphism of $Y$ onto $Y_{\beta}$.
Furthermore, for any $a\in A$,
$$
\rho(\theta(a))={}_C \la \theta(a) \, , \, \theta(1) \ra={}_A \la \theta(a) \, , \, \theta(1) \ra
={}_A \la a \, , \, 1 \ra=a .
$$
Therefore $[X, Y]=[A, Y_{\beta}]$ in $\Pic(A, C)$.
\end{proof}

We return to the unital inclusion of unital $C^*$-algebras $a\in A\mapsto a\otimes I_n\in M_n (A)$.
By Lemma \ref{lem:unit},
$$
\Ker f_A =\{[A, Y_{\beta}]\in\Pic(A, M_n (A)) \, | \, \beta\in \Aut_0 (A, M_n (A))\} .
$$
By Lemma \ref{lem:exact1}
\begin{align*}
& \Ker \, \pi\, \cap\Aut_0 (A, M_n (A)) \\
& =\{\Ad(u)\in\Aut_0 (A, M_n (A)) \, | \, \text{$u$ is a unitary element in $A\cap A'$}\} ,
\end{align*}
where $\pi$ is the homomorphism of $\Aut (A, M_n (A))$ to $\Pic (A, M_n (A))$ defined in Section \ref{sec:definition}.
But $\Ker \, \pi\, \cap \Aut_0 (A, M_n (A))=\{1\}$ since $(u\otimes I_n ) a=a(u\otimes I_n )$ for any unitary element
$u\in A\cap A'$, $a\in M_n (A)$. Thus we obtain the following example.

\begin{exam} Let $A$ be a unital $C^*$-algebra satisfying the following
sequence
$$
1\longrightarrow\Int(A)\longrightarrow\Aut (A)\longrightarrow\Pic (A)\longrightarrow 1
$$
is exact. Let $n$ be a positive number with $n\ge 2$.
Then $\Pic(A, M_n (A))$ is isomorphic to a semidirect product group of $Aut_0 (A, M_n (A))$ by
$\Pic(A)$, that is,
$$
\Pic(A, M_n (A))\cong \Aut_0 (A, M_n (A))\rtimes_s \Pic(A) .
$$
\end{exam}
%\par
%\it
%Acknowledgement.
%\rm
%The author wishes to thank the referee for many valuable suggestions for improvement for
%the manuscript.

\end{document}